\documentclass[12pt]{amsart}
\usepackage{amsthm}

\usepackage{amssymb}
\usepackage{verbatim}
\usepackage[curve,matrix,arrow]{xy}
\usepackage{amsmath,amsfonts}
\usepackage{graphicx}
\usepackage{epsf}

\usepackage{floatflt}% package for floatingfigure environment

\usepackage{amssymb, graphics}

\newcommand{\mathsym}[1]{{}}

\newcommand{\thmref}[1]{Theorem~\ref{#1}}

\newcommand{\lemref}[1]{Lemma~\ref{#1}}
\newcommand{\eqnref}[1]{Equation~(\ref{#1})}

\newcommand{\figref}[1]{Figure~\ref{#1}}

  {\end{list}}

\newtheorem{theorem}{Theorem}[section]

\newtheorem{lemma}[theorem]{Lemma}

\theoremstyle{definition}

\newcommand{\vv}[1]{V(#1)}

\def\<{\langle }
\def\>{\rangle }
\newcommand{\secref}[1]{\S\ref{#1}}

\def\mod{\text{mod}}

\newcommand{\jm}{\mathrm{J}}

\newcommand{\lm}{\mathrm{L}}
\newcommand{\plm}{\mathrm{L^+}}

\begin{document}

\title[Effective Resistances of Prism Graphs]{Effective Resistances and Kirchhoff index of Prism Graphs}

%05C12,  Distance in graphs
%94C15,   	Applications of graph theory
%05C50, Graphs and linear algebra (matrices, eigenvalues, etc.)

\author{Zubeyir Cinkir}
\address{Zubeyir Cinkir\\
Department of Industrial Engineering\\
Abdullah G\"{u}l University\\
Kayseri\\
TURKEY.}
\email{zubeyirc@gmail.com}

%\address{Zubeyir Cinkir\\
%Zirve University\\
%Faculty of Education\\
%Department of Mathematics\\
%Gaziantep\\
%TURKEY.}
%\email{zubeyirc@gmail.com}

%\author{Zubeyir Cinkir}
%\address{Zubeyir Cinkir\\
%Department of Mathematics\\
%University of Georgia\\
%Athens, Georgia 30602\\
%USA}
%\email{cinkir@math.uga.edu}

%\thanks{}

\begin{abstract}
We explicitly compute the effective resistances between any two vertices of a prism graph by using circuit reductions and our earlier findings on a ladder graph. As an application, we derived a closed form formula for the Kirchhoff index of a prism graph.
We show as a byproduct that an explicit sum formula involving trigonometric functions hold by comparing our formula for the Kirchhoff index and previously known results in the literature. We also expressed our formulas in terms of certain generalized Fibonacci numbers.
%that are the values of the Chebyshev polynomials of the second kind at $2$.)
\end{abstract}

\maketitle

\section{Introduction}\label{sec introduction}
%\vskip .1 in

\begin{floatingfigure}[r]{2.4 in}
\begin{center}
%\centering
\includegraphics[scale=0.4]{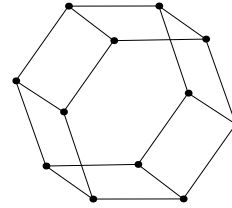}
\end{center}
\caption{Prism graph $Y_6$.} \label{fig prismgraph6}
\end{floatingfigure}
A prism graph $Y_n$ is a planar graph that looks like a circular ladder with $n$ rungs.  \figref{fig prismgraph6} illustrates $Y_6$.
$Y_n$ has $2n$ vertices and $3n$ edges. Each of its edges has length $1$, so the total length of $Y_n$ is $\ell(Y_n):=3n$.
%$g(\ga)=n+1$.

We give explicit formulas for the effective resistances between any two vertices of $Y_n$.
We consider $Y_n$ as an electrical network in which we set the resistances along edges as the corresponding edge lengths.
If we connect two ladder graphs possibly of different vertex numbers by adding four edges to their end vertices, we obtain a prism graph.
We apply circuit reductions to each of those ladder graphs by keeping their four end vertices.  This gives us a circuit reduction of the prism graph $Y_n$. The reduced $Y_n$ will have $8$ vertices. Thanks to knowing the resistance values on a ladder graph \cite{C1}, we can determine the resistance values between the vertices of the reduced $Y_n$ by utilizing the discrete Laplacian and its pseudo inverse of this reduced graph.

Let us define the sequence $G_n$ by the following recurrence relation
$$G_{n+2}=4G_{n+1}-G_{n}, \quad \text{if $n \geq 2$, and $G_0=0$, $G_1=1$}.$$

We showed that the following equalities hold for Kirchhoff index of $Y_n$ (see \thmref{thm Kirchhoff index} and \eqnref{eqn KI and Trig II} below), where $n$ is a positive integer:
\begin{equation*}
\begin{split}
Kf(Y_n)&=\frac{n(n^2-1)}{6}+ \frac{2n^2 G_{n}^2}{ G_{2n}-2G_n}
%\\&
=\frac{n(n^2-1)}{6}+\frac{n^2}{\sqrt{3}} \Big[ \frac{2}{1-(2-\sqrt{3})^{n}} -1\Big].
\end{split}
\end{equation*}
Similarly, for any positive integer $n$, we showed that the following identities of trigonometric sum hold (see \thmref{thm trig sum} and \eqnref{eqn KI and Trig II} below):
\begin{equation*}
\begin{split}
\sum_{k=0}^{n-1} \frac{1}{1+2 \sin^2({\frac{k \pi}{n}})}&=\frac{2n G_{n}^2}{ G_{2n}-2G_n}
%\\&
=\frac{n}{\sqrt{3}}\Big[ \frac{2}{1-(2-\sqrt{3})^{n}}-1 \Big].
\end{split}
\end{equation*}
The resistance values on Wheel and Fan graphs (in \cite{BG}) and Ladder graphs (in \cite{C1} and \cite{CEM}) are expressed in terms of generalized Fibonacci numbers.
Our findings for resistance values on a Prism graph are analogues of those results.

%\vskip 2 in
\section{Resistances between any pairs of vertices in $Y_n$}\label{sec resistances}
A ladder graph $L_n$ is a planar graph that looks like a ladder with $n$ rungs.
It has $2n$ vertices and $3n-2$ edges. Each of its edges has length $1$, so the total length of $L_n$ is $\ell(L_n):=3n-2$.
We obtain the prism graph $Y_n$ from $L_n$ by adding two edges connecting the end vertices on the same side.

If we delete an edge that is a rung in $Y_n$, and then contract two edges that are on each side of that rung, we obtain the prism graph $Y_{n-1}$. If we apply this process to 3-prism graph $Y_3$, we obtain the graph on the right in \figref{fig Y1andY2}. We call it 2-prism graph $Y_2$. Similarly, if we apply this process to $Y_2$, we obtain the graph on the left in \figref{fig Y1andY2}. We call it 1-prism graph $Y_1$. These graphs are the natural extension of prism graphs to the cases $n=1, \, 2$. We see that our formulas for resistance values, Kirchhoff index as well as the spanning tree formulas are also valid for these two cases (see \thmref{thm Kirchhoff index} and Equations (\ref{eqn prism2b}), (\ref{eqn spantree}) and (\ref{eqn KI and Trig II}) below).
\begin{figure}
\centering
\includegraphics[scale=0.7]{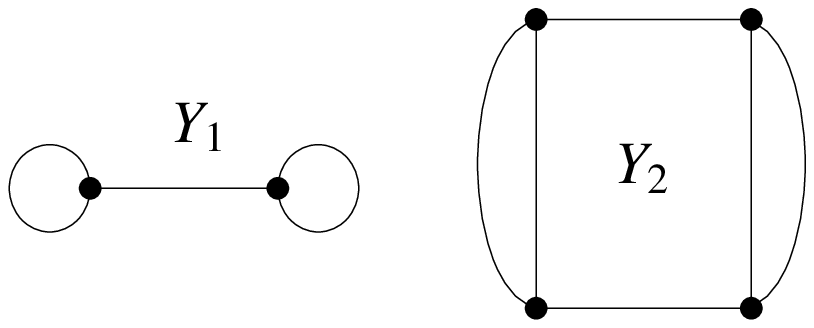} \caption{} \label{fig Y1andY2}
\end{figure}

For a prism graph $Y_n$, we label the vertices on the right and left as $\{ q_1, \, q_2, \, \cdots, q_n \}$ and $\{ p_1, \, p_2, \, \cdots, p_n \}$, respectively. This is illustrated in \figref{fig prismn}, where $2 \leq i \leq n $. We want to find the value of $r(p,q)$ for any two vertices $p$ and $q$ of $Y_n$, where $r(x,y)$ is the resistance function on $Y_n$. We implement the following strategy to do this:
\begin{itemize}
  \item Consider $Y_n$ as the union of the ladder graphs $L_{n-i+1}$ and $L_{i-1}$ as illustrated in \figref{fig 2ladders}.
  \item Apply circuit reductions on each of these ladder graphs by keeping the four end vertices as illustrated in \figref{fig ladderreduction}.
  \item Use our earlier results on a ladder graph in \cite{C1} to find the resistances between the end points of the reduced ladder graphs. Note that certain resistance values are equal to each other because of the symmetries in the ladder graphs.
  \item So far we obtain the circuit reduction of $Y_n$ by keeping its $8$ vertices $p_n$, $q_n$, $p_1$, $q_1$, $p_{i}$, $q_i$, $p_{i-1}$ and $q_{i-1}$. This is illustrated in \figref{fig 2lreduction}. Find the Moore-Penrose inverse $L^+$ of the discrete Laplacian matrix $L$ of this reduced $Y_n$. They are $8 \times 8$ matrices.
  \item Use $L^+$ and \lemref{lem disc} to find out the resistances between the $8$ vertices of $Y_n$. Again note that there are symmetries in $Y_n$ and that $i$ is an arbitrary value in $\{ 2, 3, \dots, n \}$.
\end{itemize}

\begin{figure}
\centering
\includegraphics[scale=0.7]{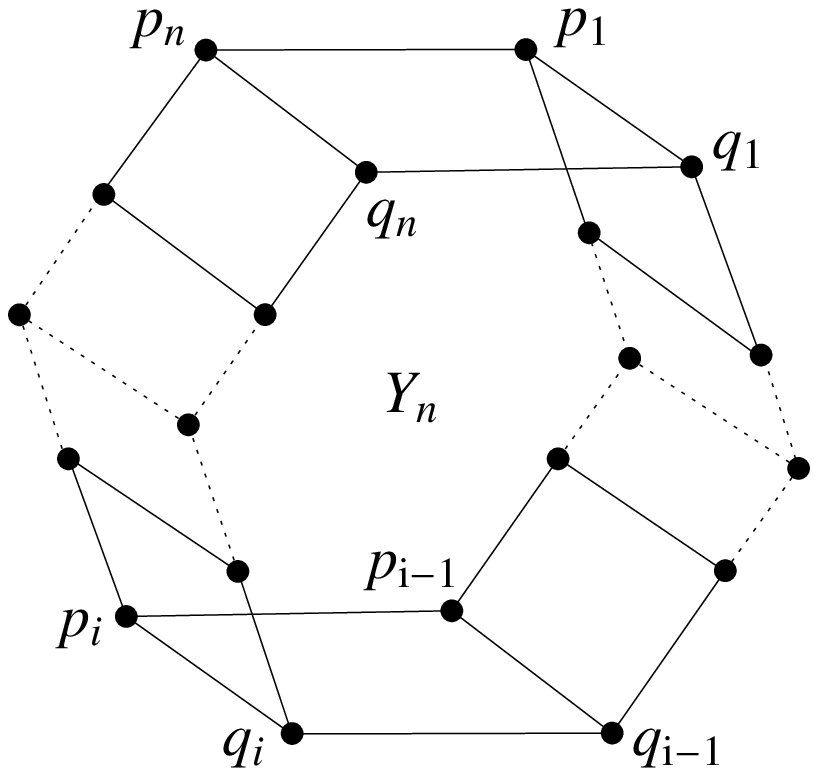} \caption{} \label{fig prismn}
\end{figure}

\begin{figure}
\centering
\includegraphics[scale=0.7]{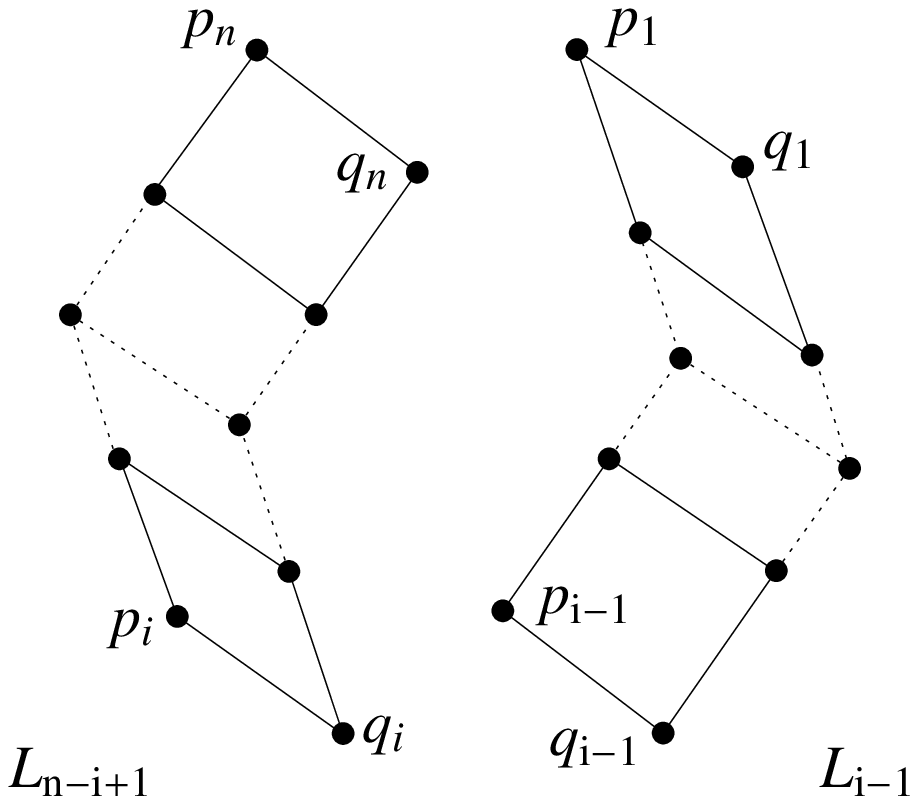} \caption{} \label{fig 2ladders}
\end{figure}

Symmetries in $Y_n$ gives the following identities of resistances:
\begin{equation}\label{eqn prism0}
\begin{split}
r(p_1,p_{i})=r(q_1,q_{i}) \quad \text{and  }  r(p_1,q_{i})=r(q_1,p_{i}), \quad  \text{for each } i \in \{  1, \dots, n \}.
\end{split}
\end{equation}

We recall that the resistance values can be expressed in terms of the entries of the pseudo inverse of the discrete Laplacian matrix.
\begin{lemma} \cite{RB2}, \cite[Theorem A]{D-M} \label{lem disc}
Suppose $G$ is a graph with the discrete Laplacian $\lm$ and the
resistance function $r(x,y)$. For the pseudo inverse $\plm$ of $\lm$, we have
$$r(p,q)=l_{pp}^+-2l_{pq}^+ + l_{qq}^+, \quad \text{for any two vertices $p$ and $q$ of $G$}.$$
\end{lemma}

\begin{figure}
\centering
\includegraphics[scale=0.7]{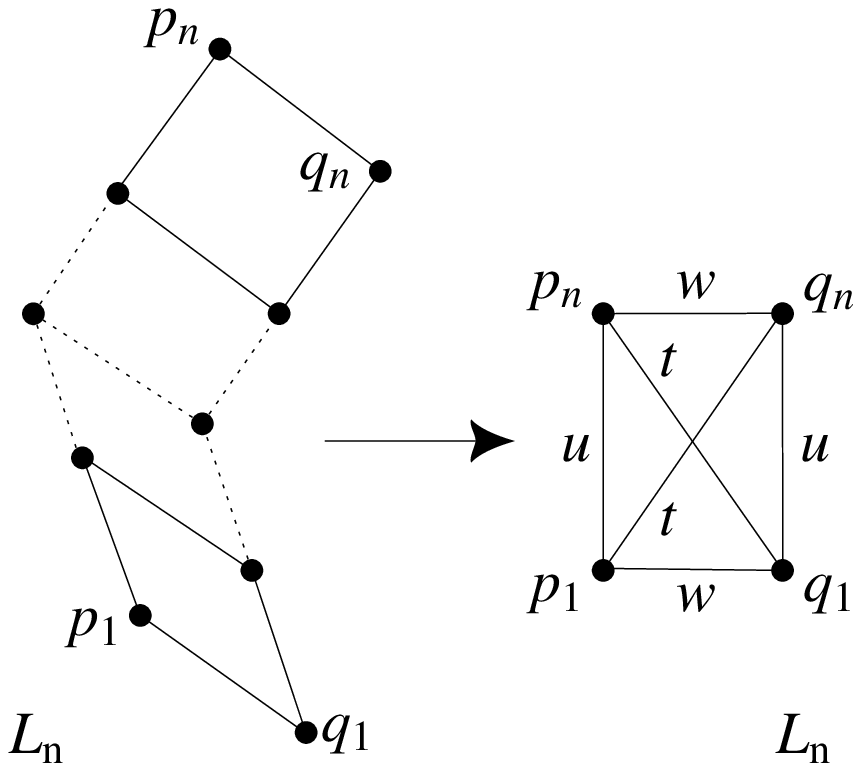} \caption{} \label{fig ladderreduction}
\end{figure}
\begin{figure}
\centering
\includegraphics[scale=0.7]{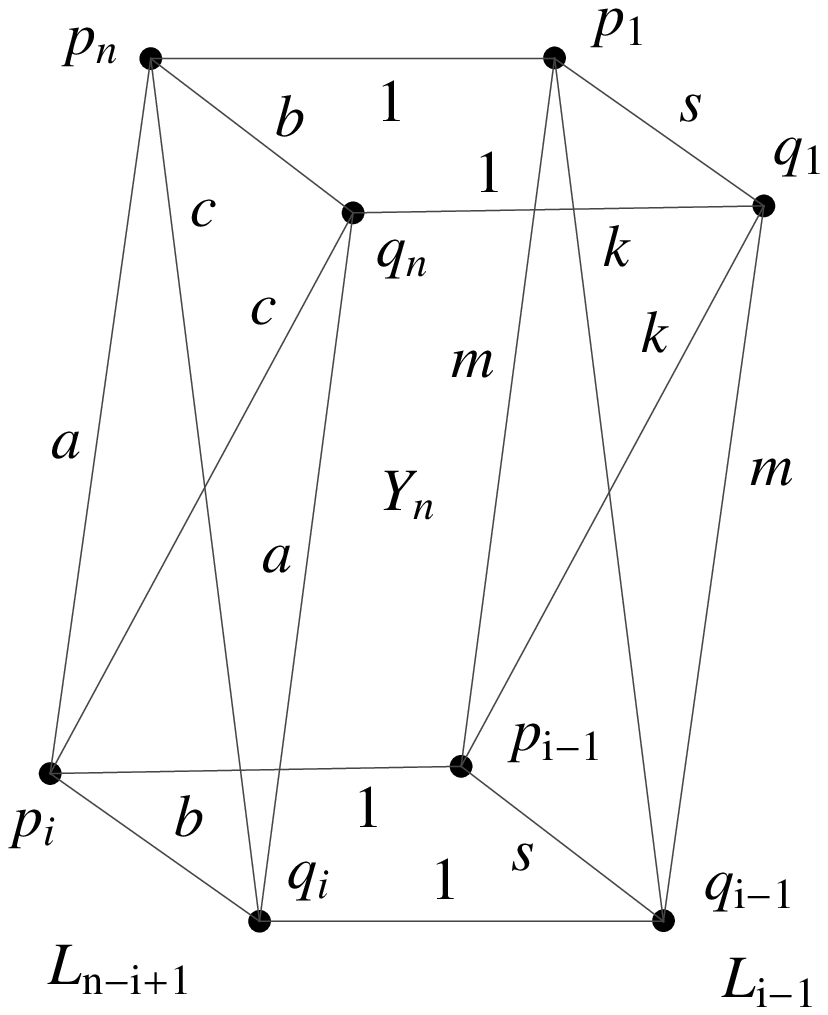} \caption{} \label{fig 2lreduction}
\end{figure}

Let $K=1+\frac{1}{k}+\frac{1}{m}+\frac{1}{s}$ and $S=1+\frac{1}{a}+\frac{1}{b}+\frac{1}{c}$, where $k$, $m$, $s$, $a$, $b$ and $c$ are the resistance values along the edges given in \figref{fig 2lreduction}. Considering the ordering of the vertices
$V=\{ p_1, \, p_{i-1}, \, p_i, \, p_n, \, q_1, \, q_{i-1}, \, q_i, \, q_n \}$, discrete Laplacian matrix $L$ of the graph (reduced $Y_n$) given in \figref{fig 2lreduction} is as follows:
$$
%\tiny
L=\left[
\begin{array}{cccccccc}
 K & -\frac{1}{m} & 0 & -1 & -\frac{1}{s} & -\frac{1}{k} & 0 & 0 \\[.15cm]
 -\frac{1}{m} & K & -1 & 0 & -\frac{1}{k} & -\frac{1}{s} & 0 & 0 \\[.15cm]
 0 & -1 & S & -\frac{1}{a} & 0 & 0 & -\frac{1}{b} & -\frac{1}{c} \\[.15cm]
 -1 & 0 & -\frac{1}{a} & S & 0 & 0 & -\frac{1}{c} & -\frac{1}{b} \\[.15cm]
 -\frac{1}{s} & -\frac{1}{k} & 0 & 0 & K & -\frac{1}{m} & 0 & -1 \\[.15cm]
 -\frac{1}{k} & -\frac{1}{s} & 0 & 0 & -\frac{1}{m} & K & -1 & 0 \\[.15cm]
 0 & 0 & -\frac{1}{b} & -\frac{1}{c} & 0 & -1 & S & -\frac{1}{a} \\[.15cm]
 0 & 0 & -\frac{1}{c} & -\frac{1}{b} & -1 & 0 & -\frac{1}{a} & S \\
\end{array}
\right].
$$
%\normalsize
Then we can compute the Moore-Penrose inverse $L^+$ of $L$ by using \cite{MMA} with
the following formula (see \cite[ch 10]{C-S}):
\begin{equation}\label{eqn disc5}
 \plm = \big( \lm - \frac{1}{8}\jm \big)^{-1} + \frac{1}{8} \jm.
\end{equation}
where $\jm$ is of size $8 \times 8$ and has all entries $1$. Next, we use $L^+$ and \lemref{lem disc} to obtain
$r(p_1,p_{i-1})$ and $r(p_1,q_{i-1})$. In this way, we find that

\tiny
$r(p_1,p_{i-1})=\frac{P}{Q}$, where
$P=m (4 b c (m s + k (m + 2 s + m s)) +
   a^2 ((2 + b) (2 + c) m s +
      k (2 (2 + b) (2 + c) s +
         m (4 + 2 c + 4 s + c s + b (2 + c + s)))) +
   2 a (2 c (m s + k (m + 2 s + m s)) +
      b (2 (1 + c) m s + k (4 (1 + c) s + m (2 + 2 c + 2 s + c s)))))$
%\normalsize
and
%\tiny
$Q=2 (c (2 m + k (2 + m)) +
   a ((2 + c) m + k (2 + c + m))) (b (2 s + m (2 + s)) +
   a ((2 + b) s + m (2 + b + s)))
$.
%\normalsize

$r(p_1,q_{i-1})=\frac{R}{S}$, where
%\tiny
$R=k (c (4 c (2 m s + k (m + s + m s)) +
      b (4 (2 + c) m s + k (2 (2 + c) s + m (4 + 2 c + 4 s + c s)))) +
    a (b (2 + c) (2 (2 + c) m s + k ((2 + c) s + m (2 + c + 2 s))) +
      c (4 (2 + c) m s + k (2 (2 + c) s + m (4 + 2 c + 4 s + c s)))))$
%\normalsize
and
%\tiny
$S=2 (c (2 m + k (2 + m)) +
   a ((2 + c) m + k (2 + c + m))) (c (2 s + k (2 + s)) +
   b ((2 + c) s + k (2 + c + s)))
$.
\normalsize

We note that these resistance values can be expressed in a more compact form as below:
\begin{equation}\label{eqn prism1}
\begin{split}
&r(p_1,p_{i-1})=\frac{1}{2}\left(\frac{1}{\frac{1}{2+\frac{a c}{a+c}}+\frac{1}{\frac{m k}{m+k}}}+\frac{1}{\frac{1}{2+\frac{a b}{a+b}}+\frac{1}{\frac{m s}{m+s}}}\right),\\
&r(p_1,q_{i-1})= \frac{1}{2}\left(\frac{1}{\frac{1}{2+\frac{a c}{a+c}}+\frac{1}{\frac{m k}{m+k}}}+\frac{1}{\frac{1}{2+\frac{ b c}{b+c}}+\frac{1}{\frac{k s}{k+s}}}\right).
\end{split}
\end{equation}

To find the exact values of the resistances in \eqnref{eqn prism1}, we need the corresponding values from the circuit reduction of $L_{n-i+1}$ and $L_{i-1}$. Thus, we turn our attention to the circuit reduction of $L_n$ as in \figref{fig ladderreduction}.
Let us apply circuit reductions to a ladder graph $L_n$ by keeping its vertices at its bottom and top so that we obtain a complete graph on $4$ vertices. Suppose $p_1$, $q_1$ and $p_n$, $q_n$ are the vertices at the bottom and top of the ladder graph.  This is illustrated in \figref{fig ladderreduction}. Note that by the symmetry in $L_n$, we have only three distinct edge lengths in the complete graph obtained. Let us consider the ordering of the vertices $\{p_n, \, q_n, \, p_1, \, q_1 \}$. Using the notations
in \figref{fig ladderreduction}, the Laplacian matrix $M$ of the reduced graph can be given as follows:
$$
M=\left[
\begin{array}{cccc}
 \frac{1}{t}+\frac{1}{u}+\frac{1}{w} & -\frac{1}{w} & -\frac{1}{u} & -\frac{1}{t} \\[.15cm]
 -\frac{1}{w} & \frac{1}{t}+\frac{1}{u}+\frac{1}{w} & -\frac{1}{t} & -\frac{1}{u} \\[.15cm]
 -\frac{1}{u} & -\frac{1}{t} & \frac{1}{t}+\frac{1}{u}+\frac{1}{w} & -\frac{1}{w} \\[.15cm]
 -\frac{1}{t} & -\frac{1}{u} & -\frac{1}{w} & \frac{1}{t}+\frac{1}{u}+\frac{1}{w} \\
\end{array}
\right].
$$
Then we use symmetries in $L_n$, compute the Moore-Penrose inverse $M^+$ of $M$ and apply \lemref{lem disc} to write
\begin{equation}\label{eqn ladder1}
\begin{split}
&r_{L_n}(p_n,q_n)=r_{L_n}(p_1,q_1)= \frac{w (w t+u (w+2 t))}{2 (u+w) (w+t)} = \frac{1}{2} \left(\frac{w t}{w+t}+\frac{u w}{u+w}\right),\\
&r_{L_n}(p_n,p_1)=r_{L_n}(q_n,q_1)= \frac{u (2 w t+u (w+t))}{2 (u+w) (u+t)}= \frac{1}{2} \left(\frac{u w }{u+w}+\frac{u t}{u+t}\right),\\
&r_{L_n}(p_n,q_1)=r_{L_n}(q_n,p_1)= \frac{t (2 u w+(u+w) t)}{2 (u+t) (w+t)} = \frac{1}{2} \left(\frac{u t }{u+t}+\frac{w t}{w+t}\right).
\end{split}
\end{equation}
Using \eqnref{eqn ladder1}, we derive
\begin{equation}\label{eqn ladder2}
\begin{split}
&A_n:=\frac{w t}{w+t}=r_{L_n}(p_n,q_n)+r_{L_n}(p_n,q_1)-r_{L_n}(p_n,p_1), \\
&B_n:=\frac{u w}{u+w}=r_{L_n}(p_n,q_n)-r_{L_n}(p_n,q_1)+r_{L_n}(p_n,p_1), \\
&C_n:=\frac{u t}{u+t}=-r_{L_n}(p_n,q_n)+r_{L_n}(p_n,q_1)+r_{L_n}(p_n,p_1).
\end{split}
\end{equation}
As particular cases of \eqnref{eqn ladder2}, we have
\begin{align}\label{eqn ladder3}
\frac{a c}{a+c} &=C_{n-i+1},&  \frac{a b}{a+b}&=B_{n-i+1},&   \frac{b c}{b+c}&=A_{n-i+1},&  \\
\label{eqn ladder4}
\frac{m k}{m+k} &=C_{i-1},&  \frac{m s}{m+s}&=B_{i-1},&   \frac{k s}{k+s}&=A_{i-1},&
\end{align}
where we used the notations as in \eqnref{eqn prism1} (and so as in \figref{fig 2lreduction}).
In \cite{C1}, we gave explicit formulas for the resistance values between any two vertices of a ladder graph.
Next, we use \cite[Equation 12 at page 959]{C1} to write
\begin{equation}\label{eqn ladder5}
\begin{split}
&A_n= -1 - \sqrt{3} +\frac{2 \sqrt{3} }{1-(2-\sqrt{3})^n}, \\
&B_n= -1 - \sqrt{3} +\frac{2 \sqrt{3} }{1+(2-\sqrt{3})^n}, \\
&C_n=n-1.
\end{split}
\end{equation}
We use Equations (\ref{eqn prism1}), (\ref{eqn ladder3}) and (\ref{eqn ladder4}) to derive
\begin{equation}\label{eqn prism2}
\begin{split}
&r(p_1,p_{i})=\frac{1}{2}\left(\frac{1}{\frac{1}{2+C_{n-i}}+\frac{1}{C_{i}}}+\frac{1}{\frac{1}{2+B_{n-i}}+\frac{1}{B_i}}\right),\\
&r(p_1,q_{i})= \frac{1}{2}\left(\frac{1}{\frac{1}{2+C_{n-i}}+\frac{1}{C_{i}}}+\frac{1}{\frac{1}{2+A_{n-i}}+\frac{1}{A_i}}\right).
\end{split}
\end{equation}
where $2 \leq i \leq n$. Next, we use \eqnref{eqn ladder5} in \eqnref{eqn prism2b} and then work with \cite{MMA} to simplify the algebraic expressions as below:
\begin{equation}\label{eqn prism2b}
\begin{split}
&r(p_1,p_{i})=\frac{1+\left(2-\sqrt{3}\right)^n-\left(2-\sqrt{3}\right)^{n-i+1}-\left(2-\sqrt{3}\right)^{i-1}}{2\sqrt{3} \left(1-\left(2-\sqrt{3}\right)^n\right)} +\frac{(n-i+1)(i-1)}{2n},\\
&r(p_1,q_{i})= \frac{1+\left(2-\sqrt{3}\right)^n+\left(2-\sqrt{3}\right)^{n-i+1}+\left(2-\sqrt{3}\right)^{i-1}}{2\sqrt{3} \left(1-\left(2-\sqrt{3}\right)^n\right)}+\frac{(n-i+1)(i-1)}{2n}.
\end{split}
\end{equation}
where $1 \leq i \leq n$.

By using the symmetries of the graph $Y_n$, we note that for every $i$ and $j$ in $\{ 1, \, 2, \, \dots, \, n \}$ we have
\begin{equation}\label{eqn prism4}
\begin{split}
r(p_1,p_{i})=r(p_j,p_k) \quad \text{and} \quad
r(p_1,q_{i})=r(p_j,q_k),
\end{split}
\end{equation}
where $1 \leq k \leq n$ and $k \equiv j+i-1$ $\mod \, n$.

Finally, the explicit values of $r(p,q)$ between any two vertices $p$ and $q$ of the prism graph $Y_n$ can be obtained by using
%\eqnref{eqn prism0}, \eqnref{eqn prism2b} and \eqnref{eqn prism4}.
Equations (\ref{eqn prism0}), (\ref{eqn prism2b}) and (\ref{eqn prism4}).

It follows from \eqnref{eqn prism2b} that
\begin{equation}\label{eqn prism3}
\begin{split}
r(p_1,p_{i})+r(p_1,q_{i})=\frac{1}{\sqrt{3}} \Big( \frac{1+\left(2-\sqrt{3}\right)^n}{1-\left(2-\sqrt{3}\right)^n} \Big)+\frac{ (n-i+1)(i-1)}{n}.
\end{split}
\end{equation}

\section{Kirchhoff Index of $Y_n$ }\label{sec Kirchhoff index}

In this section, we obtain an explicit formula for Kirchhoff index of $Y_n$ by using our explicit formulas derived in \secref{sec resistances} for the resistances between any pairs of vertices of $Y_n$. Moreover, we obtain an interesting summation formula by combining our findings and what is known in the literature about Kirchhoff index of $Y_n$.

%Recall that Kirchhoff index of a graph $\ga$, $Kf(\ga)$, is defined \cite{KR} as follows:
%\begin{equation*}\label{eqn KIndex definition}
%\begin{split}
%Kf(\ga)=\frac{1}{2}\sum_{p,\, q \in \vv{\ga}}r(p,q).
%\end{split}
%\end{equation*}

\begin{theorem}\label{thm Kirchhoff index}
For any positive integer $n$, we have
$$
Kf(Y_n)=\frac{n(n^2-1)}{6}+\frac{n^2}{\sqrt{3}} \Big[ \frac{2}{1-(2-\sqrt{3})^{n}} -1\Big].
$$
\end{theorem}
\begin{proof}
With the notation of vertices as in \figref{fig prismn} we have
\begin{equation*}
\begin{split}
Kf(Y_n)&= \frac{1}{2} \sum_{p,\, q \in \vv{Y_n}}r(p,q), \text{ by definition \cite{KR}}. \\
&=n \sum_{i=1}^n r(p_1,p_i)+r(p_1,q_i), \text{ Equations (\ref{eqn prism0}) and (\ref{eqn prism4}).}
\end{split}
\end{equation*}
Then the result follows if we use first \eqnref{eqn prism3} and do
%Equations $(\ref{eqn prism3})$, $(\ref{eqn prism2b})$ and doing
some algebra \cite{MMA}.
\end{proof}
Alternatively, we can express the Kirchhoff index formula in \thmref{thm Kirchhoff index} as follows:
$$
%Kf(Y_n)=\frac{n}{6}\big[ n^2-1-2\sqrt{3} n \coth \big(\frac{n}{2} \ln(2-\sqrt{3}) \big) \big].
Kf(Y_n)=\frac{n(n^2-1)}{6}-\frac{n^2}{\sqrt{3}}\coth \big(\frac{n}{2} \ln(2-\sqrt{3}) \big).
$$
$Kf(Y_n)$ have rational values. For example, its values for $1 \leq n \leq 10$ are as follows:

$1$, $11/3$, $47/5$, $58/3$, $655/19$, $279/5$, $5985/71$, $2540/21$, $44193/265$, $139655/627$.

Next, we show that an interesting trigonometric sum identity hold:
\begin{theorem}\label{thm trig sum}
For any positive integer $n$, we have
$$
\sum_{k=0}^{n-1} \frac{1}{1+2 \sin^2({\frac{k \pi}{n}})}=\frac{n}{\sqrt{3}}\Big[ \frac{2}{1-(2-\sqrt{3})^{n}}-1 \Big].
$$
\end{theorem}
\begin{proof}
Prism graph $Y_n$ can be seen as the cartesian product $P_2 \Box C_n$, where $P_2$ is the path graph with $2$ vertices and $C_n$ is the cycle graph with $n$ vertices. Moreover, considering the Laplacian eigenvalues of $P_2$ and $C_n$ we see that the Laplacian eigenvalues of $Y_n$ (see \cite{LCAAE}, \cite{RB} and \cite{CDS})
are
\begin{equation}\label{eqn eigen1}
\begin{split}
\lambda_{ij}=4-2\cos{(\frac{i \pi}{2})}-2\cos{(\frac{2j \pi}{n})}, \qquad \text{where $i=0,1$ and $j=0,1,\dots,n-1$}.
\end{split}
\end{equation}
We recall that \cite[pg 644]{PBM}
\begin{equation}\label{eqn eigen2}
\begin{split}
\sum_{k=1}^{n-1} \frac{1}{ \sin^2{(\frac{k \pi}{n})}}=\frac{n^2-1}{3}.
\end{split}
\end{equation}
Now, we can express the Kirchhoff index via the eigenvalues of the  discrete Laplacian matrix of $Y_n$ \cite{KR}:
\begin{equation}\label{eqn KI Prism}
\begin{split}
Kf(Y_n)&=2n \sum_{\lambda_{ij} \neq 0} \frac{1}{\lambda_{ij}}\\
&=n \sum_{k=1}^{n-1} \frac{1}{1-\cos{(\frac{2k \pi}{n})}} + n \sum_{k=0}^{n-1} \frac{1}{2-\cos{(\frac{2k \pi}{n})}}, \quad \text{by \eqnref{eqn eigen1}},\\
&=n \sum_{k=1}^{n-1} \frac{1}{2 \sin^2{(\frac{k \pi}{n})}} + n \sum_{k=0}^{n-1} \frac{1}{1+2 \sin^2{(\frac{k \pi}{n})}}, \quad \text{using $1-\cos{(\frac{2k \pi}{n})}=2 \sin^2{(\frac{k \pi}{n})}$},\\
&=\frac{n(n^2-1)}{6}+n \sum_{k=0}^{n-1} \frac{1}{1+2 \sin^2({\frac{k \pi}{n}})}, \quad \text{by \eqnref{eqn eigen2}}.
\end{split}
\end{equation}
Then the proof is completed by combining \eqnref{eqn KI Prism} and the result in \thmref{thm Kirchhoff index}.
\end{proof}

\section{Recursive Formulations}\label{sec Recurrence}

In this section, we give recursive formulas for the resistance values obtained in \secref{sec resistances}, the Kirchhoff index of $Y_n$ and the trigonometric formula given in \thmref{thm trig sum}. As we did in \cite{C1} for Ladder graph, we use the sequence of integers $G_n$ defined by the following recurrence relation
$$G_{n+2}=4G_{n+1}-G_{n}, \quad \text{if $n \geq 2$, and $G_0=0$, $G_1=1$}.$$
We have
$$
G_{n}=\frac{(2-\sqrt{3})^{-n}-(2-\sqrt{3})^n}{2 \sqrt{3}}, \quad \text{for each integer $n \geq 0$}.
$$
The sequence $G_n$ has various well-known properties \cite{SL}.
For example, it gives the number of spanning trees of $L_n$ \cite{BP}, and the number of the spanning trees of the prism graph $Y_n$  \cite{J} is given by
\begin{equation}\label{eqn spantree}
\begin{split}
\frac{n}{2} \Big( (2+\sqrt{3})^{n} + (2-\sqrt{3})^{n} -2 \Big) = \frac{n}{2} \Big( \frac{G_{2n}}{G_n} -2 \Big).
\end{split}
\end{equation}
Let $g_n=(2-\sqrt{3})^n=\frac{1}{G_{n+1}-(2-\sqrt{3})G_n}$ for any nonnegative integer $n$. Then, for any integer $i \in \{1,2, \dots, n \}$ we have
\begin{equation*}
\begin{split}
&r(p_1,p_{i})=\frac{(n-i+1)(i-1)}{2n}+\frac{ G_{n}^2}{ G_{2n}-2G_n}-\Big(\frac{1}{4\sqrt{3}}+\frac{ G_{n}^2}{ 2G_{2n}-4G_n} \Big) (g_{n-i+1}+g_{i-1}),\\
&r(p_1,q_{i})= \frac{(n-i+1)(i-1)}{2n}+\frac{ G_{n}^2}{ G_{2n}-2G_n}+\Big(\frac{1}{4\sqrt{3}}+\frac{ G_{n}^2}{ 2G_{2n}-4G_n} \Big) (g_{n-i+1}+g_{i-1}).
\end{split}
\end{equation*}
Similarly, the other resistance values can be expressed in terms of $G_n$ by using the symmetries in $Y_n$ like \eqnref{eqn prism0} and \eqnref{eqn prism4}.

Here are how we can express the results given in \thmref{thm Kirchhoff index} and \thmref{thm trig sum} in terms of $G_n$:
\begin{equation}\label{eqn KI and Trig II}
\begin{split}
Kf(Y_n)=\frac{n(n^2-1)}{6}+ \frac{2n^2 G_{n}^2}{ G_{2n}-2G_n} \qquad \text{and } \quad
\sum_{k=0}^{n-1} \frac{1}{1+2 \sin^2({\frac{k \pi}{n}})}=\frac{2n G_{n}^2}{ G_{2n}-2G_n}.
\end{split}
\end{equation}

%%\tiny
%%\small
%%\scriptsize
%\footnotesize

%\newpage
\textbf{Acknowledgements:} This work is supported by Abdullah Gul University Foundation of Turkey.

%The author would like to thank the anonymous referees and
%editor  for  their  helpful  comments  and  suggestions.   The  project  is  supported  by
%Abdullah Gul University Foundation of Turkey.

\end{document}